\documentclass[11pt]{amsart}
\usepackage{amssymb}
\usepackage{amsmath}
\usepackage{setspace}
\usepackage{amsthm}
\usepackage{epsfig,color,colordvi,wasysym}
\usepackage{graphicx}
\usepackage{amsfonts}
\usepackage{transparent}

\newtheorem{theorem}{Theorem}

\newtheorem{lemma}[theorem]{Lemma}

\newtheorem{comment}[theorem]{Comment}

\newcommand\commentout[1]{}

\newcommand\Def[1]{\emph{#1}}

\def\th{^{\text{th}}}

\newcommand{\A}{{\bar{A}}}

\newcommand{\Z}{\mathbb{Z}}

\def\l{\lambda}
\def\0{\mathbf{0}}
\def\1{\mathbf{1}}

\begin{document}

\title[Variations on a Generating-Function Theme]{Variations on a Generating-Function Theme: Enumerating Compositions with Parts Avoiding an Arithmetic Sequence}  

\author{Matthias Beck}
\author{Neville Robbins}
\address{Department of Mathematics, San Francisco State University, San Francisco, CA 94132, USA}
\email{[mattbeck,nrobbins]@sfsu.edu}


% 05A17 Partitions of integers
% 11P83 Partitions; congruences and congruential restrictions

\date{4 June 2014}

\thanks{We thank Steffen Eger and two anonymous referees for numerous helpful comments and suggestions. M.\ Beck's research was partially supported by the US National Science Foundation (DMS-1162638).}

\maketitle

\begin{abstract}
A \Def{composition} of a positive integer $n$ is a $k$-tuple $(\l_1, \l_2, \dots, \l_k) \in \Z_{ > 0 }^k$ such that $n = \l_1 + \l_2 + \dots + \l_k$.
Our goal is to enumerate those compositions whose parts $\l_1, \l_2, \dots, \l_k$ avoid a fixed arithmetic sequence. When this sequence is given by the even integers
(i.e., all parts of the compositions must be odd), it is well known that the number of compositions is given by the Fibonacci sequence. A much more recent theorem
says that when the parts are required to avoid all multiples of a given integer $k$, the resulting compositions are counted by a sequence given by a Fibonacci-type
recursion of depth $k$. We extend this result to arbitrary arithmetic sequences.
Our main tool is a lemma on generating functions which is no secret among experts but deserves to be more widely known.
\end{abstract}

\vspace{20pt}
\small
\begin{quote} 
{\it Life is the twofold internal movement of composition and decomposition at once general and continuous.} 
Henri de Blainville (1777--1850)
\end{quote} 
\normalsize

%--------------------------------------

\section{Introduction}

A \Def{composition} of a positive integer $n$ is a $k$-tuple $(\l_1, \l_2, \dots, \l_k) \in \Z_{ > 0 }^k$ such that
\[
  n = \l_1 + \l_2 + \dots + \l_k \, .
\] 
The integers $\l_1, \l_2, \dots, \l_k$ are the \Def{parts} and $k$ is the \Def{length} of the composition. Our goal is to enumerate all compositions of a given integer
for which the parts come from or avoid a fixed set, in our case formed by an arithmetic sequence. This goal is merely a variation on a theme that is well known to the
experts (see, e.g., \cite{gesselli,heubachmansourcompositionpartsset,hoggattlind,knopfmacherrobbins,milson,moserwhitney}), so that our paper has a definite expository flavor; we mainly wish to exhibit an approach to enumerating compositions that we find particularly elegant.

Enumeration results on integer compositions and \emph{partitions} (for which we do not distinguish $k$-tuples that are permutations of each other) form a classic body
of mathematics going back to at least Euler, including numerous applications throughout mathematics and some areas of physics. The books
\cite{andrewstheoryofpartitions,andrewseriksson,heubachmansour} serve as good introductions to this area of study.

Given a set $A \subseteq \Z_{ >0 }$, let $c_A(n)$ denote the number of compositions of $n$ (of any length) with parts in $A$.
Our point of departure is the case where $A$ consists of odd integers:

\begin{theorem}\label{thm:cayley}
If $A$ is the set of odd positive integers then $c_A(n)$ equals the $n\th$ Fibonacci number $f_n$, defined recursively (as usual) through
\begin{align*}
  f_1 &= f_2 := 1 \\
  f_j &:= f_{ j-1 } + f_{ j-2 } \ \text{ for } \ j \ge 3 \, .
\end{align*}
\end{theorem}

At times it will be more natural to consider compositions of an integer whose parts avoid a given set $A$, and so we define $\A := \Z_{ >0 } \setminus A$. 
For example, Theorem \ref{thm:cayley} says that, when $A$ is the set of even positive integers then $c_\A(n)$ equals the $n\th$ Fibonacci number.

It is not clear who first proved Theorem \ref{thm:cayley}. The earliest reference we are aware of is \cite{hoggattlind} but we suspect that the theorem has been known earlier.
Cayley's collected works \cite[p.~16]{cayleycollectedpapers10} contains the result that $c_{ \{ j \in \Z : \, j \ge 2 \} } (n)$ equals the $n\th$ Fibonacci number, but
the fact that $c_{ \{ j \in \Z : \, j \ge 2 \} } (n) = c_{ \{ 2j+1 : \, j \ge 0 \} } (n)$ is
nontrivial (see, e.g., Sill's recent bijective proof \cite{sillscompositions}).

Theorem \ref{thm:cayley} virtually begs to be extended. Perhaps surprisingly, the most
natural generalization appeared only recently~\cite{robbinstribonacci,robbinstribonaccigeneralized}:

\begin{theorem}\label{thm:robbins}
Fix a positive integer $k$ and let $A$ be the set of all positive multiples of $k$.
Then $c_\A(n)$ is given by the sequence $(f_n)$ defined recursively through
\begin{align*}
  f_j &:= 2^{ j-1 } \ \text{ for } \ 1 \le j \le k-1 \\
  f_k &:= 2^{ k-1 } - 1 \\
  f_j &:= f_{ j-1 } + f_{ j-2 } + \dots + f_{ j-k } \ \text{ for } \ j > k \, .
\end{align*}
\end{theorem}

Naturally, Theorem \ref{thm:cayley} is the special case $k=2$ of Theorem~\ref{thm:robbins}. The instances $k=3$ and $k=4$ give rise to \emph{Tribonacci numbers}
\cite[Sequence A001590]{sloaneonlineseq} and \emph{Tetranacci numbers} \cite[Sequence A001631]{sloaneonlineseq}, respectively.

Our goal is to give short proofs of these results using a basic but powerful tool---Lemma \ref{lem:moserwhitney} below---that deserves to be more widely known. (Thus
our real goal is to spread the word.) This tool will naturally give rise to further generalizations and closed formulas for composition counting functions. We denote the
\Def{generating function} for the counting function $c_A(n)$ by $C_A(x) := 1 + \sum_{ n \ge 1 } c_A(n) \, x^n$.

\begin{lemma}\label{lem:moserwhitney}
$ \displaystyle
  C_A(x) = \frac{ 1 }{ 1 - \sum_{ m \in A } x^m } \, .
$
\end{lemma}

The earliest reference for this lemma we are aware of is Feller's book \cite[p.~311]{feller} on probability theory, which was first published in
1950. The earliest combinatorics paper that includes Lemma~\ref{lem:moserwhitney} seems to be~\cite{moserwhitney}.

Lemma \ref{lem:moserwhitney} follows immediately from dissecting the generating function $C_A(x)$ according to how many parts a composition has: since the generating function for all compositions with exactly $j$ parts in $A$ equals $\left( \sum_{ m \in A } x^m \right)^j$, 
\begin{equation}\label{eq:moserwhitney}
  C_A(x) 
  = 1 + \sum_{ j \ge 1 } \left( \sum_{ m \in A } x^m \right)^j
  = \frac{ 1 }{ 1 - \sum_{ m \in A } x^m } \, .
\end{equation}
We note that Lemma \ref{lem:moserwhitney} is equivalent to the (equally simple) fact that $c_A(n) = \sum_{m \in A} c_A(n-m)$.

We could now leave it up to the reader as a (fun) exercise to derive Theorems \ref{thm:cayley} and \ref{thm:robbins} from Lemma \ref{lem:moserwhitney},
including the challenge to find more results of the kind. We will give proofs and further results, such as Theorem \ref{thm:arithmseq} below, but the reader who would like
to experience the charm of Lemma~\ref{lem:moserwhitney} first hand should now take pencil, paper, and a good cup of coffee, before reading on.

The following sample result is novel but we think of it merely as a variation of the theme of using Lemma \ref{lem:moserwhitney} to count
compositions, as will hopefully become apparent in the next section.

\begin{theorem}\label{thm:arithmseq}
Fix positive integers $k$ and $m$ with $m < k$, and let $A := \left\{ m + jk : \, j \in \Z_{ \ge 0 } \right\}$.
Then $c_\A(n)$ is given by the sequence $(f_n)$ defined recursively through
\begin{align*}
  f_j &:= 2^{ j-1 } \ \text{ for } \ 1 \le j \le m-1 \\
  f_j &:= 2^{ j-1 } - 2^{ j-m } \ \text{ for } \ m \le j \le k \\
  f_j &:= f_{ j-1 } + \dots + f_{ j-m+1 } + f_{ j-m-1 } + \dots + f_{ j-k+1 } + 2 f_{ j-k } \ \text{ for } \ j > k \, .
\end{align*}
\end{theorem}

Note that we could let $m=k$, as long as the recursion is interpreted correctly, which yields Theorem \ref{thm:robbins}; in this sense, Theorem \ref{thm:arithmseq}
can be viewed as a generalization of Theorem~\ref{thm:robbins}.

To conclude this introduction section, we remark that the recent literature has seen
more sophisticated applications of the above sequence constructions to general
combinatorial structures; see, e.g.,~\cite{flajoletsedgewick}.

%--------------------------------------

\section{Proofs}

To warm up, we start with the case $A = \Z_{ >0 }$. In this case Lemma \ref{lem:moserwhitney} says
\begin{align*}
  C_A(x)
  = \frac{ 1 }{ 1 - \sum_{ m \ge 1 } x^m } 
  = \frac{ 1 }{ 1 - \frac{ x }{ 1-x } } 
  = \frac{ 1-x }{ 1-2x } 
  = 1 + \frac{ x }{ 1-2x } 
  = 1 + \sum_{ n \ge 1 } 2^{ n-1 } x^n ,
\end{align*}
confirming that there are $2^{ n-1 }$ compositions of the positive integer $n$.
The above line contains essentially all ingredients we need for more complicated sets~$A$.

\begin{proof}[Proof of Theorem \ref{thm:robbins}]
Fix a positive integer $k$ and let $A$ be the set of all positive multiples of $k$.
By Lemma \ref{lem:moserwhitney},
\[
  C_\A(x)
  = \frac{ 1 }{ 1 - \frac{ x + x^2 + \dots + x^{ k-1 } }{ 1 - x^k }  } 
  = \frac{ 1 - x^k }{ 1 - x - x^2 - \dots - x^k } \, .
\]
Developing the right-hand side into a power series $1 + \sum_{ n \ge 1 } f_n \, x^n$ gives
\[
  \left( 1 - x - x^2 - \dots - x^k \right) \left( 1 + \sum_{ n \ge 1 } f_n \, x^n \right) = 1 - x^k ,
\]
and comparing coefficients yields the identities
\begin{align*}
  f_1 &= 1 \\
  f_2 &= f_1 + 1 = 2 \\
  f_3 &= f_2 + f_1 + 1 = 4 \\
      &\ \ \vdots \\
  f_{ k-1 } &= f_{ k-2 } + f_{ k-3 } + \dots + f_1 + 1 = 2^{ k-2 } \\
  f_k &= f_{ k-1 } + f_{ k-2 } + \dots + f_1 = 2^{ k-1 } - 1 \\
  f_j &= f_{ j-1 } + f_{ j-2 } + \dots + f_{ j-k } \ \text{ for } \ j > k \, . \qedhere
\end{align*}
\end{proof}

\begin{proof}[Proof of Theorem \ref{thm:arithmseq}]
Fix positive integers $k$ and $m$ with $m < k$, and let $A := \left\{ m + jk : \, j \in \Z_{ \ge 0 } \right\}$.
By Lemma \ref{lem:moserwhitney},
\[
  C_\A(x)
  = \frac{ 1 }{ 1 - \frac{ x + \dots + x^{ m-1 } + x^{ m+1 } + \dots + x^k }{ 1 - x^k }  } 
  = \frac{ 1 - x^k }{ 1 - x - \dots - x^{ m-1 } - x^{ m+1 } - \dots - x^{ k-1 } - 2 x^k } \, .
\]
Developing the right-hand side into a power series as in our proof of Theorem \ref{thm:robbins} yields the coefficient identities
\begin{align*}
  f_j &= 2^{ j-1 } \ \text{ for } \ 1 \le j \le m-1 \\
  f_j &= 2^{ j-1 } - 2^{ j-m } \ \text{ for } \ m \le j \le k \\
  f_j &= f_{ j-1 } + \dots + f_{ j-m+1 } + f_{ j-m-1 } + \dots + f_{ j-k+1 } + 2 f_{ j-k } \ \text{ for } \ j > k \, . \qedhere
\end{align*}
\end{proof}

%--------------------------------------

\section{Closed Formulas}

The classic formula
\begin{equation}\label{eq:fibformula}
  \frac {1}{\sqrt 5} \left( \frac{ 1 + \sqrt 5 }{ 2 } \right)^n - \frac {1}{\sqrt 5} \left( \frac{ 1 - \sqrt 5 }{ 2 } \right)^n
\end{equation}
for the $n\th$ Fibonacci number follows at once from a partial-fraction expansion of the Fibonacci generating function $\frac{ x }{ 1 - x - x^2 }$.
Since all generating functions that appear in our paper evaluate to rational functions, we can obtain closed formulas just as effortlessly for any of the composition
counting functions we discussed. We give a sample.

\begin{theorem}\label{thm:closedform}
Let $A = \left\{ 1+3j : \, j \in \Z_{ \ge 0 }  \right\}$, and let
\[
  \alpha \approx 0.6572981061 \ \text{ and } \ \beta, \gamma \approx -0.5786490531 \pm 0.6525757633 \, i
\]
be the three roots of $1-x^2-2x^3$. Then
\begin{equation}\label{eq:closedform}
  c_\A(n) = \frac{ 1 + \alpha }{ (2 + 6 \alpha) } \left( \frac 1 \alpha \right)^n +
\frac{ 1 + \beta }{ (2 + 6 \beta) } \left( \frac 1 \beta \right)^n + \frac{ 1 +
\gamma }{ (2 + 6 \gamma) } \left( \frac 1 \gamma \right)^n \, .
\end{equation}
\end{theorem} 

\begin{proof}
By Lemma \ref{lem:moserwhitney},
\[
  C_\A(x) = \frac{ 1 }{ 1 - \frac{ x^2 + x^3 }{ 1 - x^3 } } = \frac{ 1 - x^3 }{ 1 - x^2 - 2 x^3 } \, .
\]
This rational function comes with the partial-fraction expansion
\[
  C_\A(x) = \frac{ \frac{ 1 - \alpha^3 }{ -2 \alpha - 6 \alpha^2 } }{ x - \alpha } 
       + \frac{ \frac{ 1 - \beta^3 }{ -2 \beta - 6 \beta^2 } }{ x - \beta } 
       + \frac{ \frac{ 1 - \gamma^3 }{ -2 \gamma - 6 \gamma^2 } }{ x - \gamma } 
\]
which, using $1 - \alpha^3 = \alpha^2 + \alpha^3$ and the analogous relations for $\beta$ and $\gamma$, gives
\[
  C_\A(x) = \frac{ 1 + \alpha }{ (2 + 6 \alpha) } \, \frac{ 1 }{ 1 - \frac x \alpha }
       + \frac{ 1 + \beta }{ (2 + 6 \beta) } \, \frac{ 1 }{ 1 - \frac x \beta }
       + \frac{ 1 + \gamma }{ (2 + 6 \gamma) } \, \frac{ 1 }{ 1 - \frac x \gamma } \, .
\]
The result now follows from expanding the geometric series.
\end{proof}

One of the charming consequences of the formula \eqref{eq:fibformula} for the Fibonacci sequence is that the term $- \frac {1}{\sqrt 5} \left( \frac{ 1 - \sqrt 5 }{ 2
} \right)^n$ converges to zero so quickly that we can compute the $n\th$ Fibonacci number as the nearest integer to $\frac {1}{\sqrt 5} \left( \frac{ 1 + \sqrt 5 }{ 2
} \right)^n$.
Unfortunately, the situation with the formula for $c_\A(n)$ presented in Theorem \ref{thm:closedform} is not quite as friendly, since $\alpha$, $\beta$, and $\gamma$
all have absolute value less than 1, and thus each of the three terms in \eqref{eq:closedform} grows exponentially with $n$.

Repeating the steps in the proof of Theorem \ref{thm:closedform} for the case $A = \left\{ 2+3j : \, j \in \Z_{ \ge 0 }  \right\}$ gives a similar picture: here we
need to consider the three roots
\[
  \alpha \approx 0.5897545123 \ \text{ and } \ \beta, \gamma \approx -0.2948772562 \pm 0.8722716255 \, i
\]
of $1-x-2x^3$. Again all of them have absolute value less than 1, and so each of the three terms in the associated composition counting function 
\[
  c_\A(n) = \frac{ 1 + \alpha^2 }{ (1 + 6 \alpha^2) } \left( \frac 1 \alpha
\right)^n + \frac{ 1 + \beta^2 }{ (1 + 6 \beta^2) } \left( \frac 1 \beta \right)^n + \frac{ 1 + \gamma^2 }{ (1 + 6 \gamma^2) } \left( \frac 1 \gamma \right)^n
\]
grows exponentially with $n$.

On the other hand, if we consider the case $A = \left\{ 3j : \, j \in \Z_{ > 0 }  \right\}$, we need the roots 
\[
  \alpha \approx 0.5436890127 \ \text{ and } \ \beta, \gamma \approx -0.7718445063 \pm 1.1151425080 \, i
\]
of $1-x-x^2-x^3$. Here the complex roots have absolute value larger than 1, and so the associated composition counting function 
\[
  c_\A(n) = \frac{ 1 + \alpha }{ (1 + 2 \alpha + 3 \alpha^2) } \left( \frac 1
\alpha \right)^n + \frac{ 1 + \beta }{ (1 + 2 \beta + 3 \beta^2) } \left( \frac
1 \beta \right)^n + \frac{ 1 + \gamma }{ (1 + 2 \gamma + 3 \gamma^2) } \left( \frac 1 \gamma \right)^n
\]
is well approximated by the first term; in fact, $c_\A(n)$ equals the nearest integer to
$\frac{ 1 + \alpha }{ (1 + 2 \alpha + 3 \alpha^2) } \left( \frac 1 \alpha \right)^n$, for any $n > 0$. Of course, the error decreases exponentially; it is smaller than 1\% already for $n=4$. 

Table \ref{tab:examplevalues} gives the first twenty values for the three composition counting functions mentioned in this section; here we use the notation
$c_{ 3, j } (n)$ for the number of compositions of $n$ with parts not congruent to $j$ (mod 3).

\begin{table}[htb]
\begin{center}
\begin{tabular}{|r|r|r|r|}
\hline
$n$ & $ c_{ 3, 1 } (n) $ & $ c_{ 3, 2 } (n) $ & $ c_{ 3, 0 } (n) $ \\
\hline
 1 & 0 & 1 & 1 \\
\hline
 2 & 1 & 1 & 2 \\
\hline
 3 & 1 & 2 & 3 \\
\hline
 4 & 1 & 4 & 6 \\
\hline
 5 & 3 & 6 & 11 \\
\hline
 6 & 3 & 10 & 20 \\
\hline
 7 & 5 & 18 & 37 \\
\hline
 8 & 9 & 30 & 68 \\
\hline
 9 & 11 & 50 & 125 \\
\hline
10 & 19 & 86 & 230 \\
\hline
11 & 29 & 146 & 423 \\
\hline
12 & 41 & 246 & 778 \\
\hline
13 & 67 & 418 & 1431 \\
\hline
14 & 99 & 710 & 2632 \\
\hline
15 & 149 & 1202 & 4841 \\
\hline
16 & 233 & 2038 & 8904 \\
\hline
17 & 347 & 3458 & 16377 \\
\hline
18 & 531 & 5862 & 30122 \\
\hline
19 & 813 & 9938 & 55403 \\
\hline
20 & 1225 & 16854 & 101902 \\
\hline
\end{tabular}
\end{center}
\caption{Evaluations of the composition counting functions avoiding arithmetic sequences modulo~3.}\label{tab:examplevalues}
\end{table}

%--------------------------------------

\section{Concluding Remarks}

There exist bivariate versions of all results we have discussed so far, where one keeps track of both the number $n$ whose compositions we want to count and the
number $m$ of parts in the compositions. Let $c_A(n,m)$ denote the number of compositions of $n$ with precisely $m$ parts in the set $A$, and let
\[
  C_A(x,y) := \sum_{ n, m \ge 0 } c_A(n,m) \, x^n y^m , 
\]
where we set $c_A(0,0) := 1$. Of course, this means we can recover our previous generating function as $C_A(x) = C_A(x,1)$.
The accompanying version of Lemma \ref{lem:moserwhitney}, which one easily verifies by inserting a $y$ into the beginning of the large parenthesis in
\eqref{eq:moserwhitney}, seems to have first appeared in \cite{hoggattlind}.

\begin{lemma}\label{eq:moserwhitneyext}
$ \displaystyle
  C_A(x,y) = \frac{ 1 }{ 1 - y \sum_{ m \in A } x^m } \, .
$
\end{lemma}

For example, when $A$ is the set of all positive odd integers,
\[
  C_A = \frac{ 1 }{ 1 - y \frac{ x }{ 1-x^2 } } = \frac{ 1 - x^2 }{ 1 - x^2 - xy } = 1 + \frac{ xy }{ 1 - x^2 - xy } \, .
\]
This yields a recurrence similar to the one given by Pascal's triangle, and in fact, $c_A(n,m)$ are
binomial coefficients, as can also be easily proved bijectively.

The usefulness of Lemma \ref{lem:moserwhitney} is not restricted to the cases treated so far. For example, for $A = \Z_{ \ge 2 }$ we obtain
\[
  C_A(x)
  = \frac{ 1 }{ 1 - \frac{ x^2 }{ 1-x } }
  = \frac{ 1-x }{ 1 - x - x^2 }
  = 1 + \frac{ x^2 }{ 1 - x - x^2 } \, ,
\]
which is the rational generating function for the Fibonacci series appended by a constant term of 1; this confirms Cayley's result mentioned in the introduction.
One might just as well consider, e.g., $A = \Z_{ \ge 3 }$:
\[
  C_A(x)
  = \frac{ 1 }{ 1 - \frac{ x^3 }{ 1-x } }
  = \frac{ 1-x }{ 1 - x - x^3 }
  = 1 + \frac{ x^3 }{ 1 - x - x^3 } \, ,
\]
whose corresponding composition counting function $c_A(n)$ is hence given by a third-order linear recurrence.

As a final example, we give a simple generating-function proof of the following result of Zeilberger
\cite{zeilbergercompositions}, which was inspired by Sills \cite{sillscompositions}.

\begin{theorem}
Fix positive integers $a$ and $b$. Then the number of compositions of $n$ with parts $a$ and $b$ equals the number
of compositions of $n+a$ with parts in $\{ a + bj : \, j \ge 0 \}$ (and thus, by symmetry, also the number of
compositions of $n+b$ with parts in $\{ aj + b : \, j \ge 0 \}$).
\end{theorem}

This theorem generalizes the fact that both compositions with odd parts (the case $a=1$, $b=2$) and compositions
with parts greater than 1 (the case $a=2$, $b=1$) are counted by the Fibonacci numbers.

\begin{proof}
Let $A = \{ a, b \}$ and $B = \{ a + bj : \, j \ge 0 \}$. By Lemma \ref{lem:moserwhitney},
\[
  C_B(x)
  = \frac{ 1 }{ 1 - \frac{ x^a }{ 1 - x^b } }
  = \frac{ 1 - x^b }{ 1 - x^a - x^b }  
  = 1 + \frac{ x^a }{ 1 - x^a - x^b }
  = 1 + x^a \, C_A(x) \, . \qedhere
\]
\end{proof}

All of the results in this section are merely further examples of the treasure trove that Lemma~\ref{lem:moserwhitney} provides.

%--------------------------------------

\bibliographystyle{amsplain}  
\bibliography{bib}  

\end{document}